\documentclass[11pt,letterpaper]{amsart}
\usepackage{amsmath}
\usepackage{csquotes}
\usepackage{calrsfs}
\DeclareMathAlphabet{\pazocal}{OMS}{zplm}{m}{n}
\newcommand{\Xx}{\pazocal{X}}

\usepackage{enumerate}
\usepackage{csquotes}
\usepackage{amsxtra}
\newcommand{\quotes}[1]{``#1''}
\usepackage{scalerel}
\usepackage{hyperref}
\usepackage{enumitem}
\usepackage{floatrow}
\usepackage[new]{old-arrows}
\usepackage{amssymb}
\usepackage{tikz-cd}

\theoremstyle{plain}

\newtheorem{thm}{Theorem}[section]

\newtheorem{Theo}[thm]{Theorem}
\newtheorem{fact}[thm]{Fact}

\newtheorem{Prop}[thm]{Proposition}

\newtheorem*{Theorem}{Theorem}
\newtheorem{cor}[thm]{Corollary}

\newtheorem{Conj}{Conjecture}
\newtheorem{question}[thm]{Question}

\theoremstyle{definition}

\newtheorem{Lemm}[thm]{Lemma} 

\newtheorem{definition}[thm]{Definition}
\theoremstyle{remark}

\newtheorem{Obs}[thm]{Observation}
\newtheorem{rem}[thm]{Remark}

\newtheorem{example}[thm]{Example}

\newtheorem{ex}[thm]{Exercise}

\newcommand{\nc}{\newcommand}

\nc{\bp}{\begin{Prop}}
\nc{\ep}{\end{Prop}}
\nc{\ssn}{\pagebreak \section}
\nc{\df}{{\bf Definition}\ }
\nc{\bl}{\begin{Lemm}}
\nc{\el}{\end{Lemm}}
\nc{\Pic}{\text{Pic}}
\nc{\bex}{\begin{ex} \rm}
\nc{\eex}{\end{ex}}
\nc{\bt}{\begin{Theo}}
\nc{\et}{\end{Theo}}
\nc{\bq}{\begin{question}}
\nc{\eq}{\end{question}}
\nc{\bc}{\begin{cor}}
\nc{\ec}{\end{cor}}
\nc{\bob}{\begin{Obs}}
\nc{\eob}{\end{Obs}}
\nc{\N}{\mathbb{N}}
\nc{\Q}{\mathbb{Q}}
\nc{\Z}{\mathbb{Z}}
\nc{\Ss}{{\mathbb{S}}}
\nc{\Cc}{{\mathbb{C}}}
\nc{\F}{{\mathbb{F}}}
\nc{\Oo}{\mathcal{O}}
\nc{\Qq}{\mathbb{Q}}
\nc{\ulim}{\text{ulim}\ }
\nc{\Hom}{\text{Hom}}
\nc{\Ext}{\text{Ext}}
\nc{\Tor}{\text{Tor}}
\nc{\Ob}{\text{Ob}}
\nc{\id}{\text{id}}
\nc{\ad}{\text{ad}}
\nc{\an}{\text{an}}
\nc{\rig}{\text{rig}}
\nc{\Spa}{\text{Spa}}
\nc{\ZR}{\text{ZR}}
\nc{\Hilb}{\text{Hilb}}
\nc{\supp}{\text{supp}}
\nc{\Spf}{\text{Spf}}

\renewcommand{\mkbegdispquote}[2]{\itshape}
 \nc{\Spec}{{\mathop{\operatorname{\rm Spec}}}}
  \nc{\spec}{{\mathop{\operatorname{\rm Spec}}}}

\usepackage{yhmath}

\usepackage[utf8]{inputenc}

\title{Perfectoid $C_i$ transfer}
\author{Konstantinos Kartas}
\thanks{The author thanks the Max Planck Institute for Mathematics in Bonn for its hospitality and financial support. He also thanks the Alexander von Humboldt Foundation for its financial support during the revision of this work.}

\begin{document}

\maketitle
\begin{abstract}
We prove a perfectoid analogue of the Ax-Kochen theorem on zeros of $p$-adic forms: Given $d\in \N$, there is a finite totally ramified extension $E/\Q_p$ such that every untilt of $\F_p(\!(t^{1/p^{\infty}})\!)$ containing $E$ is $C_2(d)$. We also prove a similar result for the existence of rational points in rationally connected varieties over perfectoid field extensions of $\Q_p^{ur}$.
\end{abstract}

\section{Introduction}
A field $k$ is called $C_i$ ($i\in \N$) if every non-constant homogeneous polynomial  \linebreak
$f(X_0,...,X_n)\in k[X_0,...,X_n]$ of degree $d$ with $ d^i\leq n$ has a non-trivial zero over $k$. Given $d\in \N$, we also say that $k$ is $C_i(d)$ if it satisfies the above condition restricted to polynomials of degree $\leq d$. Thus, knowing that $k$  is $C_i$ (or $C_i(d)$) gives us some information about the polynomial equations which admit non-trivial solutions in $k$.
\subsection*{Ax-Kochen theorem}
In 1936, E. Artin conjectured that $\Q_p$ is a $C_2$ field. Strong evidence for this conjecture was given by Lang, who showed that $\F_p(\!(t)\!)$ is $C_2$. Ax-Kochen \cite{AK12} famously proved an asymptotic version of Artin's conjecture using model theory: 
\begin{Theorem}[Ax-Kochen]
For every $d\in \N$, there exists $N(d)\in \N$ such that $\Q_p$ is $C_2(d)$ for every $p>N(d)$. 
\end{Theorem}
They deduce this from a general transfer principle between $\Q_p$ and $\F_p(\!(t)\!)$, proved independently by Ershov \cite{Ershov}, which says that for any sentence  $\varphi$ in the language of rings, there exists $N(\varphi)\in \N$ such that 
$$\varphi \mbox{ holds in }\Q_p \iff \varphi \mbox{ holds in } \F_p(\!(t)\!)$$ 
for all $p>N(\varphi)$. This, in turn, is a consequence of a more fundamental theorem about henselian valued field of equal characteristic $0$, saying that two such valued fields are elementarily equivalent precisely when their value groups and residue fields are. The asymptotic transfer between $\Q_p$ and $\F_p(\!(t)\!)$ can  easily be deduced using ultraproducts.

This transfer principle has several other applications. A recent one, due to Duesler-Knecht \cite{duesler}, is related to the problem of whether $\Q_p^{ur}$ is geometrically $C_1$. Recall that a \textit{geometrically $C_1$ field} is one which admits a rational point on every smooth projective separably rationally connected variety (see \S \ref{geomc1}).
This problem is a prominent open case of the Lang-Manin conjecture which predicts that every $C_1$ field is geometrically $C_1$, see \cite{esnaultc1}. We note that $\Q_p^{ur}$ is known to be $C_1$ by a theorem of Lang \cite{langc1}. Duesler-Knecht prove the following asymptotic result: 
\begin{Theorem}[Duesler-Knecht]
For every $f \in \Q[t]$, there exists $N(f)\in \N$ such that for every prime $p>N(f)$, every rationally connected variety over $\Q_p^{ur}$ with Hilbert polynomial $f$ has a rational point. 
\end{Theorem}
Their proof relies on the Graber-Harris-Starr theorem \cite{graber}, which asserts that function fields of complex curves are geometrically $C_1$. More precisely, one needs the local version, due to Colliot-Th\'el\`ene's \cite{colliot}, that $\Cc(\!(t)\!)$ is geometrically $C_1$. 
The result of Duesler-Knecht was later improved by Pieropan \cite{pieropan}, replacing the dependence on the Hilbert polynomial by dependence only on the dimension of the variety:
\begin{Theorem}[Pieropan]
For every $d\in \N$, there exists $N(d)\in \N$ such that $\Q_p^{ur}$ is geometrically $C_1$ up to dimension $d$ for every $p>N(d)$.
\end{Theorem}
It is an interesting question whether anything can be said about such problems for fixed $p$. Somewhat surprisingly, $\Q_p$ is not $C_2$ for any $p$, due to examples of Terjanian \cite{Terj1,Terj2}.
As for the Lang-Manin problem for $\Q_p^{ur}$, there is only a partial result due to Esnault-Levine-Wittenberg \cite{esnaultindex}, which says that every rationally connected variety over $\Q_p^{ur}$ has a zero-cycle of $p$-power degree.  
The purpose of this paper is to illustrate how recent work of Jahnke and the author \cite{JK} can be used to prove analogous results to the ones presented above, for fixed $p$ but high ramification.

\subsection*{Main results}
We write $\F_p(\!(t^{1/p^{\infty}})\!)$ for the completed perfection of $\F_p(\!(t)\!)$.
Our first main result is a perfectoid analogue of the Ax-Kochen theorem:
\bt \label{firstmain}
For every $d\in \N$, there exists a finite totally ramified extension $E/\Q_p$ such that every untilt of $\F_p(\!(t^{1/p^{\infty}})\!)$ containing $E$ is $C_2(d)$. 
\et

The result remains valid for any algebraic extension of $\Q_p$ containing $E$ whose completion tilts to $\F_p(\!(t^{1/p^{\infty}})\!)$. We therefore obtain plenty of algebraic extensions of $\Q_p$ which are $C_2(d)$. 
For example, given  $d\in \N$, there is $E/\Q_p$ as above such that  $E(\pi^{1/p^{\infty}})$ is $C_2(d)$, for any uniformizer $\pi$  of $E$. In \S \ref{RSC}, we prove a generalization of the above result  for rationally simply connected varieties, building on Starr-Xu \cite{starrxu}.

Our second main result was motivated by a question of Esnault from 2014 (cf. \S 5.3 \cite{esnaultc1}), who asked whether anything interesting can be said about the Lang-Manin problem for $\Q_p^{ur}$ after passing to a perfectoid field extension. This is a natural question, since $\overline{\F}_p(\!(t^{1/p^{\infty}})\!)$ is geometrically $C_1$, as follows from de Jong–Starr \cite{dejongstarr}.
\bt \label{secondmain}
For every $d\in \N$, there exists  a finite extension $E/\Q_p^{ur}$ such that every untilt of $\overline{\F}_p(\!(t^{1/p^{\infty}})\!)$ containing $E$ is geometrically $C_1$ up to dimension $d$. 
\et 
We do not know if the finite extension $E/\Q_p$ is necessary in the above results. It may well be the case that every untilt of $\F_p(\!(t^{1/p^{\infty}})\!)$ is $C_2$ and every untilt of $\overline{\F}_p(\!(t^{1/p^{\infty}})\!)$ is geometrically $C_1$. As for the latter, we show in \S \ref{conditional} that it follows from a standard conjecture on degenerations of rationally connected varieties: 
\begin{Conj} \label{kollcoll}
Let $R$ be a DVR with fraction field $K$ and residue field $k$. Let $X$ be a flat, projective $R$-scheme such that the generic fiber $X_K$ is smooth and separably rationally connected. Then, the special fiber $X_k$ admits a $k$-morphism from a smooth projective separably rationally connected $k$-variety.
\end{Conj}
If in addition $X$ is regular and $X_k$ is a strict normal crossings divisor of $X$, it is expected that one of the irreducible components of $X_k$ is separably rationally connected. The conjecture was originally proposed by Koll\'ar with $R$ being the local ring at a $k$-point of a smooth curve over a characteristic $0$ field (see Question 11 \cite{kollar}). The extension to arbitrary DVRs is a suggestion of Colliot-Th{\'e}l{\`e}ne (see Suggestions 7.9 \cite{colliot}). Conjecture \ref{kollcoll} is true in equal characteristic $0$ due to Hogadi-Xu \cite{hogadi} (see also Th\'eor\`eme 7.15 \cite{colliot}). We refer the reader to \S 7 \cite{colliot} for partial results in arbitrary characteristic.

Assuming this conjecture, we prove a perfectoid \quotes{geometric $C_1$} transfer:
\bt \label{conditionperf}
Assume Conjecture \ref{kollcoll}. Let $K$ be a perfectoid field with tilt $K^{\flat}$ which is geometrically $C_1$. Then $K$ is geometrically $C_1$.
\et  
For example, assuming Conjecture \ref{kollcoll}, every untilt of $\overline{\F}_p(\!(t^{1/p^{\infty}})\!)$ is geometrically $C_1$. As before, the result remains valid for algebraic extensions of $\Q_p^{ur}$ which tilt to $\overline{\F}_p(\!(t^{1/p^{\infty}})\!)$, like $\Q_p^{ur}(p^{1/p^{\infty}})$. This implies  that every smooth projective rationally connected variety $X$ over $\Q_p^{ur}$ has a closed point of $p$-power degree. As noted earlier, Esnault-Levine-Wittenberg \cite{esnaultindex} unconditionally proved a weaker statement, namely that such an $X$ has a zero-cycle of $p$-power degree. 
 
\subsection*{Leitfaden}
We sketch the proof of Theorem \ref{firstmain}, which illustrates many of the key ideas in this paper. Let $\kappa$ be an untilt of $\kappa^{\flat}=\F_p(\!(t^{1/p^{\infty}})\!)$. 
\subsubsection*{Step 1: Model theory of perfectoid fields}
The main result of \cite{JK} says that, grosso modo, $\kappa^{\flat}$ behaves like a residue field of $\kappa$. To make this precise, we need to replace $\kappa$ with some non-principal ultrapower, say $K$. Then, there is valuation $v$ on $K$, whose residue field $k$ is an elementary extension of $\kappa^{\flat}$, as shown below:
\[
  \begin{tikzcd}[column sep=4.5em, row sep=2.5em]
    K \arrow[r, "v"]  & k \\
    \kappa \arrow[r, dashed, bend left=10, "\flat"] \arrow[u, "\preceq", no head]   & \kappa^{\flat} \arrow[u, "\preceq", no head]   
     \end{tikzcd}
\]
The valuation $v$ is obtained as  a suitable coarsening of the non-standard valuation on $K$ and has very nice properties; it is henselian defectless with divisible value group. In particular, the valued field $(K,v)$ is \textit{tame}, a property we will exploit later on.
\subsubsection*{Step 2: Construction of a $C_2$ field}
We now cook up another valued field $(K',v')$ with value group and residue field that are elementarily equivalent to those of $(K,v)$, but which is also $C_2$ by  design. Start with any maximal totally ramified extension $E_{\infty}$ of $\Q_p$. Such a field is $C_1$ by \cite{KK4}. We equip the rational function field $E_{\infty}(t)$ with the Gauss valuation $u$, which has value group $\Q$ and residue field $\F_p(t)$. 
We now take an algebraic extension $(K',v')$ of $(E_{\infty}(t),u)$ which is henselian defectless with value group $\Q$ and residue field $k'=\F_p(t)^{h,\text{perf}}$. This finishes our construction of $(K',v')$. 
We claim that 
$$k\equiv k' \mbox{ and } \Gamma\equiv \Gamma' $$
The statement about the residue fields follows from another result from \cite{JK}, which says that $\F_p(t)^{h,\text{perf}}$ is an elementary substructure of $\F_p(\!(t^{1/p^{\infty}})\!)$. For the value groups, we use the standard fact that any two non-trivial divisible ordered abelian groups have the same theory. 
Finally, note that $K'$ is indeed $C_2$ by Lang's transition theorem, since it has transcendence degree $1$ over the $C_1$ field $E_{\infty}$. 
\subsubsection*{Step 3: Model theory of tame fields}
We now use the model theory of tame fields, due to F.-V. Kuhlmann \cite{Kuhl}, to compare $(K,v)$ and $(K',v')$. This theory does not quite imply that $(K,v)$ and $(K',v')$ are elementarily equivalent, even though their residue fields and value groups are elementarily equivalent. However, it does imply that any given sentence $\varphi$ which holds in $K'$ will also hold in $K$, provided that the relative algebraic closure of $\Q_p$ in $K$ contains a sufficiently large finite subextension $E/\Q_p$ of $E_{\infty}/\Q_p$ (depending on $\varphi$). For any given $d\in \N$, we apply this to the sentence which expresses the $C_2(d)$ property, thereby obtaining that $K$, and hence $\kappa$, satisfies $C_2(d)$.
\subsection*{Organization of the paper} 
In \S \ref{model}, we present some preliminary results from the model theory of valued fields. In \S \ref{C1}, we recall some facts about geometrically $C_1$ fields and prove Theorems \ref{secondmain} and \ref{conditionperf}. In \S \ref{C2}, we prove Theorem \ref{firstmain} and also a generalization in the context of rationally simply connected varieties.

\section{Model theory of valued fields} \label{model}
We refer the reader to the article by F.-V. Kuhlmann \cite{Kuhl} on the model theory of tame fields and to the one by Jahnke and the author \cite{JK} on the model theory of perfectoid fields.
\subsection{Model theory of tame fields}

\begin{definition}
Let $(K,v)$ be a henselian valued field with value group $\Gamma$ and residue field $k$. A finite valued field extension $(K',v')$ of $(K,v)$, with value group $\Gamma'$ and residue field $k'$, is said to be \textit{tame} if the following are satisfied:
\begin{enumerate}[label=(\roman*)]
\item  If $p=\text{char}(k)>0$, then $p\nmid [\Gamma':\Gamma]$. 
\item The residue field extension $k'/k$ is separable.
\item The extension $(K',v')/(K,v)$ is \textit{defectless}, i.e., 
$$[K':K]=[\Gamma':\Gamma]\cdot [k':k]$$

\end{enumerate}
We say that $(K,v)$ is \textit{tame} if every finite valued field extension of $(K,v)$ is tame. 
\end{definition}

Given first-order $L$-structures $M$ and $N$ with a common substructure $A$, we use the notation $M\equiv_A N$ to say that the structures $M$ and $N$ are elementarily equivalent in the language $L$ enriched with constant symbols for each element of $ A$. We write \(L_{\mathrm{rings}}=\{0,1,+,\cdot\}\) for the language of rings, \(L_{\mathrm{oag}}=\{0,+,<\}\) for the language of ordered abelian groups, and
\(L_{\mathrm{val}}=L_{\mathrm{rings}}\cup\{\mathcal O\}\) for the language of valued fields, where $\Oo$ is a unary predicate for the valuation ring. 
\begin{fact}[Theorem 7.1 \cite{Kuhl}] \label{kuhlake}
Let $(K_0,v_0)$ be a defectless valued field and let $(K,v), (K',v')$ be two tame fields extending $(K_0,v_0)$. Assume that $\Gamma/\Gamma_0$ is torsion-free and $k/k_0$ is separable. Then: 
$$(K,v)\equiv_{(K_0,v_0)} (K',v')\mbox{ in }L_{\text{val}}\iff k\equiv_{k_0} k'\mbox{ in }L_{\text{rings}} \mbox{ and } \Gamma\equiv_{\Gamma_0} \Gamma'\mbox{ in }L_{\text{oag}}$$

\end{fact}
In equal characteristic, one can take $K_0$ to be the prime field (either $\Q$ or $\F_p$) equipped with the trivial valuation. This yields the following:
\begin{fact} [Theorem 1.4 \cite{Kuhl}] \label{akekuhl}
Let $(K,v)$ and $(K',v')$ be two equal characteristic tame fields. Then: 
$$(K,v)\equiv (K',v')\mbox{ in }L_{\text{val}}\iff k\equiv k'\mbox{ in }L_{\text{rings}} \mbox{ and } \Gamma\equiv \Gamma'\mbox{ in }L_{\text{oag}}$$
\end{fact}
Beware that the analogous statement fails in mixed characteristic. Indeed, there are numerous examples of mixed characteristic tame fields $(K,v)$ and $(K',v')$ with $k= k'$ and $\Gamma = \Gamma'$, yet their algebraic parts over $\Q$ are non-isomorphic, implying that they are not elementarily equivalent; see \cite{KuhlAns}. 
Upcoming work of Ketelsen provides more nuanced counterexamples where also the algebraic parts of $K$ and $K'$ are isomorphic. Nevertheless, a compactness argument gives the following more refined statement, which plays a central role throughout:
\bp \label{tamecor}
Let $\varphi$ be a sentence in the language of valued fields and $E_{\infty}$ be a maximal totally ramified extension of  $\Q_p$. Then there exists a finite subextension $E/\Q_p$ of $E_{\infty}/\Q_p$ such that for any two tame fields $(K,v)$ and $(K',v')$ extending $(E,v_p)$ with $k\equiv k'$ in $L_{\text{rings}}$ and $(\Gamma,vp)\equiv (\Gamma',v'p)$ in $L_{\text{oag}}$, we have that 
$$\varphi\mbox{ holds in }(K,v) \iff \varphi \mbox{ holds in }(K',v') $$
\ep 
\begin{proof}
Suppose otherwise. 
Write $E_{\infty}=\bigcup_{i\in \N} E_i$, where 
$$E_1\subseteq E_2 \subseteq ...$$ 
is an increasing chain of finite totally ramified extensions of $\Q_p$. By assumption, for each $E_i/\Q_p$ there exist tame valued extensions $(K_i,v_i)$ and $(K_i',v_i')$ of $E_i$ with residue fields $k_i,k_i'$ and value groups $\Gamma_i,\Gamma_i'$ such that
\[
  k_i \equiv k_i' \quad\text{and}\quad (\Gamma_i, v_i p) \equiv (\Gamma_i', v_i' p),
\]
and $\varphi$ holds in $(K_i,v_i)$ but not in $(K_i',v_i')$. Let $U$ be a non-principal ultrafilter on $\N$ and consider the ultraproducts 
$$(K,v)=\prod_{i\in \N} (K_i,v_i)/U \mbox{ and }(K',v')=\prod_{i\in \N} (K_i',v_i')/U $$
Since tame fields form an elementary class (see \S 7.1 \cite{Kuhl}),  both $(K,v)$ and $(K',v')$ are tame. Moreover, they both extend $(E_{\infty},v_p)$ and we have $k\equiv k'$ and $(\Gamma,vp)\equiv (\Gamma',v'p)$ by \L o\'s. Note that $(E_{\infty},v_p)$ is defectless, since every finite extension comes entirely from the residue field, and $\Gamma/\Q$ is clearly torsion-free. By Fact \ref{kuhlake}, we get that 
$$(K,v)\equiv_{(E_{\infty},v_p)} (K',v')$$ 
which is contrary to the fact that $\varphi$ holds in $(K,v)$ but not in $(K',v')$.
\end{proof}
The following two lemmas are  well-known to experts but we include proofs for completeness.
\bl [Theorem 2.14 \cite{KuhlVal}]\label{construction1}
Let $(K_0,v_0)$ be a valued field with value group $\Gamma_0$ and residue field $k_0$. Let $\Gamma$ be an ordered abelian group extending $\Gamma_0$ such that $\Gamma/\Gamma_0$ is a torsion group and $k/k_0$ be an algebraic extension. Then there exists an algebraic extension $(K,v)$ of $(K_0,v_0)$ with value group $\Gamma$ and residue field $k$.
\el 
\begin{proof}
We begin by extending the residue field to $k$. Suppose first that $k/k_0$ is a finite extension with minimal polynomial $f(X)\in k_0[X]$. Let $F(X)\in K_0[X]$ be a monic lift and $a$ be a root of $F(X)$ in $\overline{K_0}$. By Lemma 3.21 \cite{vdd}, the valuation $v_0$ extends uniquely to $K_0(a)$, with value group $\Gamma_0$ and residue field $k$. The case of an arbitrary algebraic extension follows by transfinite recursion, using the previous construction for the successor step and taking unions for the limit step.

We now extend the value group to $\Gamma$. Suppose that $\gamma\in \Gamma \backslash \Gamma_0$ is such that $p\cdot \gamma \in \Gamma_0$ for some prime $p$. Let $b_0\in K_0$ be such that $v_0b_0=p\cdot \gamma_0$ and $b \in \overline{K_0}$ be such that $b^p=b_0$. By Lemma 5.6 \cite{vdd}, the valuation $v_0$ extends uniquely to $K_0(b)$, with value group $\Gamma_0+\Z \gamma$ and residue field $k$. We can write $\Gamma=\bigcup_{i \in \kappa} \Gamma_i$, where $\kappa$ is an ordinal and 
$$\Gamma_0\subseteq \Gamma_1 \subseteq ... \subseteq \Gamma_i\subseteq \Gamma_{i+1} \subseteq ...$$
is an increasing chain of ordered abelian groups such that $\Gamma_{i+1}=\Gamma_i+\Z \gamma_i $ with $\gamma_i\in \Gamma_{i+1}\backslash \Gamma_i$ and $p\cdot \gamma_i\in \Gamma_i$ for some prime $p$. We again apply transfinite recursion, using the previous construction for the successor step, to construct an algebraic extension $(K,v)$ of $(K_0,v_0)$ with value group $\Gamma$ and residue field $k$.
\end{proof}

\bl \label{construction2}
Let $(K_0,v_0)$ be a valued field with value group $\Gamma_0$ and residue field $k_0$ of characteristic $p$. Let $\Gamma$ be a $p$-divisible ordered abelian group such that $\Gamma/\Gamma_0$ is a torsion group and $k$ be a perfect field such that $k/k_0$ is algebraic. Then there exists an algebraic extension $(K,v)/(K_0,v_0)$ such that $(K,v)$ is tame with value group $\Gamma$ and residue field $k$.
\el 
\begin{proof}
By Lemma \ref{construction1}, there is an algebraic extension $(K_1,v_1)$ of $(K_0,v_0)$ with value group $\Gamma$ and residue field $k$. Let $(K,v)$ be a maximal immediate algebraic extension of $(K_1,v_1)$. By Theorem 3.2 \cite{Kuhl}, we get that $(K,v)$ is tame and has all the desired properties. 
\end{proof}

\subsection{Model theory of perfectoid fields}
We refer the reader to \S 3 \cite{Scholze} for basic facts about perfectoid fields and tilting. 
\begin{definition}
A \textit{perfectoid field} is a complete valued field $(K,v)$ of residue characteristic $p>0$ such that the value group is a dense subgroup of $\mathbb{R}$ and the Frobenius map $\varphi : \mathcal{O}_K/(p) \to \mathcal{O}_K/(p):x\mapsto x^p$ is surjective.
\end{definition}
We can transform any perfectoid field $K$ to a perfectoid field $K^{\flat}$ of characteristic $p$, called the \textit{tilt} of $K$. As a set, we have $K^{\flat}=\{ (x_n)_{n\in \omega} \in K^{\omega}: x_{n+1}^p=x_n\}$. We define  multiplication in $K^{\flat}$ coordinatewise 
$$(x_n)_{n\in \omega} \cdot (y_n)_{n\in \omega}=(x_n\cdot y_n)_{n\in \omega} $$ 
Addition in $K^{\flat}$ is a bit more involved and is described  by the following rule
$$ (x_n)_{n\in \omega} + (y_n)_{n\in \omega} = (z_n)_{n\in \omega} \mbox{ where }z_n=\lim_{m\to \infty} (x_{n+m} + y_{n+m})^{p^m}$$ 
We have the \textit{sharp} map from $K^{\flat}$ to $K$, which simply picks out $x_0$, namely
$$\sharp: K^{\flat}\to K:(x_n)_{n\in \omega} \mapsto x_0 $$
It is clearly a multiplicative morphism. We define a valuation $v^{\flat}$ on $K^{\flat}$ by 
$$v^{\flat}(x)=v(x^{\sharp})$$ 
\begin{fact} [Theorem 6.2.3 \cite{JK}]\label{JKimproved}
Let $(K,v)$ be a perfectoid field and $\varpi\in \mathfrak{m}_v\backslash \{0\}$. Let $U$ be a non-principal ultrafilter on $\N$ and $(K_U,v_U)$ be the corresponding ultrapower. Let $S\subseteq \Oo_{v_U}$ be the set of elements with infinitesimal valuation, i.e., smaller than any positive rational multiple of $v_U\varpi$, and $\Oo_w=S^{-1}\Oo_{v_U}$. Then: 
\begin{enumerate}[label=(\Roman*)]
\item The valued field $(K_U,w)$ is tame with divisible value group. 

\item The tilt $(K^{\flat},v^{\flat})$ embeds elementarily into $(k_w,\overline{v})$, where $\overline{v}$ is the induced valuation of $v_U$ on $k_w$.
\end{enumerate} 
\end{fact}
Let $F$ be a perfectoid field of characteristic $p$. An \textit{untilt} of $F$ is a pair $(K,\iota)$, where $K$ is a perfectoid field and $\iota: K^{\flat} \to F$ is an isomorphism. Fargues-Fontaine \cite{FF} construct a one dimensional noetherian regular $\Q_p$-scheme  $X^{\text{FF}}$ (but not of finite type over $\Q_p$) such that the closed points of $X^{\text{FF}}$ are in natural bijection with the set of untilts of $F$ modulo an action of Frobenius. More generally, for any finite extension $E/\Q_p$ whose residue field is contained in the residue field of $F$, Fargues-Fontaine define $X^{\text{FF}}_E$ which parametrizes untilts of $F$ containing $E$. 

We show that the untilts of $F$ approach the same theory as $E$ become large:
\bp \label{perfprop}
Let $\varphi$ be a sentence in the language of rings. Let $F$ be a perfectoid field of characteristic $p$ and $E_{\infty}$ be a maximal totally ramified extension of $\Q_p$. Then there exists a finite subextension $E/\Q_p$ of $E_{\infty}/\Q_p$ such that for every untilt $K$ of $F$ containing $E$, we have that 
$$\varphi\mbox{ holds in }K \iff \varphi \mbox{ holds in }K' $$
for any tame field $K'$ extending $E$ with divisible value group and residue field elementarily equivalent to $F$.
\ep 
\begin{proof}
By Proposition \ref{tamecor}, there is a finite totally ramified extension $E/\Q_p$ such that any two tame valued fields extending $(E,v_p)$ with divisible value
group and residue field elementarily equivalent to $F$ agree on the truth
value of $\varphi$. Now let $K$ be an untilt of $F$ containing $E$. Let $U$ be a non-principal ultrafilter on $\N$ and $(K_U,w)$ be as in Fact \ref{JKimproved}. Then $(K_U,w)$ is tame with divisible value group, its residue field is elementarily equivalent to $F$, and it extends $(E,v_p)$, so by Proposition~2.5 we have
\[
  \varphi \text{ holds in } K_U \quad\Longleftrightarrow\quad
  \varphi \text{ holds in } K'
\]
for any tame field $K'$ extending $E$ with divisible value group and residue
field elementarily equivalent to $F$. On the other hand, $K_U$ is an ultrapower of $K$ as a pure field, so
\[
  \varphi \text{ holds in } K \quad\Longleftrightarrow\quad
  \varphi \text{ holds in } K_U.
\]
Combining the two equivalences gives the desired statement.
\end{proof}

\begin{rem}

\begin{enumerate}[label=(\roman*)]
\item In particular, for any untilts $K$ and $K'$ of $F$ containing $E$ (depending on $\varphi$), we have that 
$$\varphi\mbox{ holds in }K \iff \varphi \mbox{ holds in }K' $$
\item  We note that, in general, the untilts of $F$ are not elementarily equivalent to each other. For example, there are continuum many untilts of $\F_p(\!(t^{1/p^{\infty}})\!)$ with non-isomorphic algebraic parts (see Proposition 3.6.9 \cite{KK1}).
\end{enumerate}
\end{rem}

The following result will also be important: 
\begin{fact} [Corollary 5.2.2 \cite{JK}] \label{perfectJK}
Let $k$ be a perfect field of characteristic $p$. The perfect hull of $(k(t)^h,v_t)$ is an elementary substructure of $(k(\!(t^{1/p^{\infty}})\!),v_t)$.
\end{fact}

\section{Perfectoid $C_1$ transfer} \label{C1}
After recalling some definitions and facts about rationally connected varieties, we prove our main theorem about untilts of $\overline{\F}_p(\!(t^{1/p^{\infty}})\!)$. We then present some more general results which are conditional on Conjecture \ref{kollcoll}.

\subsection{Geometrically $C_1$ fields } \label{geomc1}

\begin{definition}[ IV, \S 3 \cite{kollarbook}]
A geometrically irreducible $k$-variety $X$ is called \textit{rationally connected} (resp. \textit{separably rationally connected}) if there is a $\overline{k}$-variety $B$ and a morphism $F:B\times \mathbb{P}^1\to X$ such that  the induced morphism
$$B\times \mathbb{P}^1\times \mathbb{P}^1 \to X\times X: (b,t,t')\mapsto (F(b,t),F(b,t'))$$
is dominant (resp. dominant and separable). 
\end{definition}
In other words, there is an algebraic family of proper rational curves such that for almost any $(x,x')\in X\times X$, there is a curve in the family joining $x$ and $x'$. 
\begin{definition}
\begin{enumerate}[label=(\roman*)]
\item A field $k$ is called \textit{geometrically }$C_1$ if every smooth projective separably rationally connected $k$-variety has a $k$-rational point. 

\item A field $k$ is \textit{geometrically }$C_1$\textit{ up to dimension $d$ }if every smooth projective separably rationally connected $k$-variety of dimension $\leq d$ has a $k$-rational point. 

\end{enumerate}
\end{definition}
The term \quotes{geometrically $C_1$} is due to Koll\'ar, see \cite{hogadi}. Recall from \cite{KK4} that geometrically $C_1$ fields form an elementary class. We now show that for the class of geometrically $C_1$ up to dimension $d$, a single sentence suffices, at least in characteristic $0$. 
This essentially follows from the work of 
Pieropan \cite{pieropan}, as we explain below.
The main theorem relies crucially on results from the minimal model program \cite{MMP}.
\begin{fact}[Theorem 1.3 \cite{pieropan}] \label{pierofact}
Let $k$ be a field of characteristic $0$. For every $d\in \N$, the
following statements are equivalent:
\begin{enumerate}[label=(\roman*)]
\item $k$ is geometrically $C_1$ up to dimension $d$.
\item Every terminal Fano variety of dimension $\leq d$
over $k$ has a $k$-point.
\end{enumerate}
\end{fact}
\begin{rem}
The first clause in \cite{pieropan} is about proper rather than projective varieties; this does not make a difference by Lemma 3.1 \cite{pieropan}. We also note that the second clause further restricts the class of Fano varieties but we will not need it. 
\end{rem}
Pieropan proves the following result with $N$ and $f_i$ a priori depending on $k$, but a slight modification of the proof shows that this  is not necessary.
\begin{fact} [cf. Proposition 3.4 \cite{pieropan}]\label{boundedness}
For any $d\in \N$, there exist $N\in \N$ and finitely many polynomials $f_1, ..., f_s \in \Q[t]$ such
that for every field $k$ of characteristic $0$ and every terminal Fano $k$-variety $X$ of dimension $ d$, there exists an embedding $X\subseteq \mathbb{P}^N_k$ such that $X$ has Hilbert polynomial $f_i$ for some $i\in  \{1,...,s\}$.
\end{fact}

\begin{proof}
It is enough to provide such $N$ and $f_i$ which work for countable fields of characteristic $0$, since every variety is defined over a countable field. For every such field $k$ of characteristic $0$, we fix an embedding $\iota_k:k\hookrightarrow \Cc$. By Birkar’s Theorem~1.1 \cite{birkarfano}, projective $d$-dimensional terminal Fano $\Cc$-varieties form a bounded family. Hence there exist a scheme $T$ of finite type over $\Cc$ and a projective morphism $ \pi \colon V \to T$ such that every $d$-dimensional terminal Fano $\Cc$-variety $X$ is isomorphic to a fiber $V_t$ for some $t \in T(\Cc)$. Using a flattening stratification of  $T$ (see Tag 0H3Y \cite{sp}), we obtain finitely many projective, flat morphisms
\[
  \pi_i  : V^i \to T^i, \qquad i = 1,\dots,s,
\]
such that every $d$-dimensional terminal Fano $\Cc$-variety appears as a fibre of some $\pi_i$. In particular, for every countable field $k$ of characteristic $0$ and every terminal Fano $k$-variety $X$ of dimension $d$, there exist $i\in\{1,\dots,s\}$ and $t\in T^i(\Cc)$ such that
\[
  V^i_t \cong X_\Cc,
\]
where $X_{\Cc}$ denotes the base change of $X$ from $k$ to $\Cc$ via $\iota_k$.
The rest of the proof is as in Proposition 3.4 \cite{pieropan}.
\end{proof}
The proof of the next result is similar to the proof of Theorem 11 \cite{duesler}. (Note that there is a small issue with their proof: The morphism $g$ should be the universal family over the Hilbert scheme rather than what is written there.)  
\bl \label{dueslersent}
For every Hilbert polynomial $f \in \Q[t]$ and $N\in \N$, there is a $\forall \exists$-sentence $\varphi$ in the language of rings such that for every field $k$ of characteristic $0$ the following are equivalent: 
\begin{enumerate}[label=(\roman*)] 

\item Every projective rationally connected $k$-variety $X\subseteq \mathbb{P}^N_k$ with Hilbert polynomial $f$ has a $k$-rational point. 

\item The sentence $\varphi$ holds in $k$.

\end{enumerate}
\el 
\begin{proof}
Let  $H_f$ be the component of the Hilbert scheme parametrizing closed subschemes of $\mathbb{P}^N$ with Hilbert polynomial $f$. Let $\pi: \mathcal{U} \to H_f$ be the universal family. 
By Proposition 2.6(a) \cite{dFF}, there is a constructible subset $Z\subseteq H_f$ such that for each $z\in H_f$ we have that 
$$\pi^{-1}(z)\mbox{ is a rationally connected }\kappa(z)\mbox{-variety} \iff z\in Z$$
where $\kappa(z)$ denotes the residue field of $Z$ at the point $z$. Since $\mathcal{U},H_f$ and  $\pi$ are of finite type over $\Q$ and $Z\subseteq H_f$ is constructible, there is a $\forall\exists$-sentence $\varphi$ in the language of rings expressing (uniformly in $k$) that for every $ z \in Z(k)$ the fiber $\pi^{-1}(z)$ has a $k$-rational point. By the universal property of  the Hilbert scheme, for every field $k$  of characteristic $0$ and every projective $k$-variety $X\subseteq \mathbb{P}^N_k$ with Hilbert polynomial $f$,
there is a pullback diagram 
\[
  \begin{tikzcd}[column sep=4.5em, row sep=2.5em]
    X \arrow[r] \arrow[d]  & \mathcal{U} \arrow[d] \\
    \Spec(k)  \arrow[r, "x"]&  H_f
     \end{tikzcd}
\]
Moreover, we have that $X$ is rationally connected if and only if $x\in Z(k)$. It easily follows that (i) and (ii) are equivalent.
\end{proof}
\begin{rem}
We note that Proposition 2.6(a) \cite{dFF}, which was used in the above proof, is stated over algebraically closed fields. This assumption is not necessary since the proof in \textit{loc.cit} relies on resolution of singularities and IV, Theorem 3.11 \cite{kollar}, both of which are valid for arbitrary fields of characteristic $0$. 
\end{rem}
\bp \label{sentencephi}
For any $d\in \N$, there is a $\forall \exists$-sentence $\varphi_d$ in the language of rings such that for every field $k$ of characteristic $0$, we have that
$$k \mbox{ is geometrically }C_1\mbox{ up to dimension }d \iff \varphi_d \mbox{ holds in } k  $$
\ep 
\begin{proof}
By Fact \ref{boundedness}, there are  $N\in \N$ and finitely many polynomials $f_1, ..., f_s \in \Q[t]$ such
that for every field $k$ of characteristic $0$ and every terminal Fano $k$-variety $X$ of dimension $\leq d$, there exists an embedding $X\subseteq \mathbb{P}^N_k$ such that $X$ has Hilbert polynomial $f_i$ for some $i\in  \{1,...,s\}$. Let $\varphi_i$ be a sentence associated to $f_i$ and $N$ as in Lemma \ref{dueslersent}. 
We claim that the conjunction $\varphi_d$ of the $\varphi_i$'s has the desired property.\\
\quotes{$\Rightarrow$}: Let $k$ be of characteristic $0$ and geometrically $C_1$ up to dimension $d$. By Lemma 3.1 \cite{pieropan}, every projective rationally connected $k$-variety (not necessarily smooth) of dimension $\leq d$  has a $k$-point. In particular, this is true for those rationally connected varieties $X\subseteq \mathbb{P}_k^N$ with Hilbert polynomial $f_i$, for each $i=1,...,s$. It follows that  $\varphi_d$ holds in $k$. \\
\quotes{$\Leftarrow$}: Let $k$ be of characteristic $0$ and suppose that  $\varphi_d$ holds in $k$. Then, for every $i=1,...,s$, every rationally connected $k$-variety with Hilbert polynomial $f_i$ has a $k$-point. Since terminal Fano varieties are rationally connected \cite{Zhang,HM}, all such $k$-varieties of dimension $\leq d$ have a $k$-point. We conclude from Fact \ref{pierofact} that $k$ is geometrically $C_1$ up to dimension $d$. 
\end{proof}

\subsection{Perfectoid $C_1$ transfer} \label{perfectoidc1transfer}
The following theorem is due to Graber-Harris-Starr \cite{graber} in characteristic $0$ and de Jong-Starr \cite{dejongstarr} in arbitrary characteristic. 
\begin{fact} [de Jong-Graber-Harris-Starr] \label{ghs}
Let $C$ be a smooth projective irreducible curve over an algebraically closed field $k$. Let $f:X\to C$ be a proper flat $k$-morphism whose generic fiber is smooth and separably rationally connected. Then $f$ has a section.  
\end{fact}
Since any smooth projective separably rationally connected variety over the function field $k(C)$ spreads out to a family $f:X \to C$ as above, and a section of $f$ yields a
$k(C)$-rational point on the generic fiber, it follows that the function field
$k(C)$ is geometrically $C_1$. It easily follows that any field of transcendence degree 1 over an algebraically closed field is geometrically $C_1$. Less trivially, the same holds for \quotes{sufficiently nice} valued fields whose residue field has transcendence degree 1 over an algebraically closed field:
\bp \label{construction3}
Let $(K,v)$ be a tame field with divisible value group and residue field $k$ which has transcendence degree $1$ over an  algebraically closed field.  

\begin{enumerate}[label=(\roman*)] 

\item If $(K,v)$ is of equal characteristic, then $K$ is geometrically $C_1$.

\item If $(K,v)$ is of mixed characteristic and $ \overline{\Q}\subseteq K$, then $K$ is geometrically $C_1$.
\end{enumerate}
\ep  
\begin{proof}
Say $k$ is algebraic over $k_0(t)$, where $k_0$ is algebraically closed.\\
(i) By Proposition 3.1.5 \cite{KK4}, geometrically $C_1$ fields form an elementary class. For any $k$ as above, it suffices to construct a single geometrically $C_1$ tame field $(K,v)$ with divisible value group and residue field $k$, since Fact \ref{akekuhl} guarantees that every other such field has the same theory. By de Jong-Graber-Harris-Starr, the rational function field $\overline{k_0(z)}(t)$ over $\overline{k_0(z)}$ is geometrically $C_1$. The same is true for its perfect hull, being a colimit of geometrically $C_1$ fields. Let $v_z$ be a valuation on $\overline{k_0(z)}$ extending the $z$-adic valuation on $k_0(z)$. We equip $\overline{k_0(z)}(t)$ with the Gauss extension $u$ of $v_z$ on $\overline{k_0(z)}$, with value group $\Q$ and residue field $k_0(t)$. By Lemma \ref{construction2}, there is an algebraic extension $(K,v)$ of $(\overline{k_0(z)}(t),u)$ such that $(K,v)$ is tame with value group $\Q$ and residue field $k$. Since tame fields are perfect, we have that $K$ is a separable algebraic extension of the perfect hull of $\overline{k_0(z)}(t)$. Since the latter is geometrically $C_1$, the same is true for $K$ by Lemma 3.1.4(ii) \cite{KK4}.\\
(ii) Applying Fact \ref{kuhlake}, with $(\overline{\Q},v_p)$ as the base field, we see that any two tame fields over $(\overline{\Q},v_p)$ with divisible value group and residue field $k$ have the same theory. It therefore suffices to construct one such tame field which is geometrically $C_1$. Start with any algebraically closed valued field $(K_0,v_0)$ with residue $k_0$ and proceed as in (i) to construct a tame field $(K,v)$ extending $(K_0,v_0)$, with divisible value group and residue field $k$, and $K/K_0$ of transcendence degree $1$. 
\end{proof}
The above result is closely related to a theorem of Starr on the existence of rational points on degenerations of rationally connected varieties over function fields of curves:
\begin{fact}[Theorem 3.10 \cite{starr2}] \label{starrfact}
Let $R$ be a prime regular DVR with fraction field $K$ and residue field $k$ which has transcendence degree $1$ over an  algebraically closed field. Let $X$ be a flat, projective $R$-scheme with $X_K$ smooth and separably rationally connected. Then $X(k)\neq \emptyset$.
\end{fact}
A \textit{prime regular} DVR is one which is weakly unramified over a DVR with finite residue field (see \S 3 \cite{starr2}). The prime regular assumption is relevant only in mixed characteristic and is similar in nature to our assumption that $ \overline{\Q}\subseteq K$. 
Let us explain how Fact \ref{starrfact} follows from Proposition \ref{construction3}, at least for perfect $k$. In fact, the argument is, in essence, quite close to Starr's original proof.
\begin{proof}[Proof sketch]
By Lemma \ref{construction2}, there is an algebraic extension $(K',v')$ of $(K,v)$ such that $(K',v')$ is tame with divisible value group and residue field $k$. Moreover, if $(K,v)$ is mixed characteristic, we can arrange that $K'$ contains $\overline{\Q}$. The point is that the assumption on the DVR implies that $K\cdot \overline{\Q}$ is a totally ramified algebraic extension of $K$ and therefore has residue field $k$.
By Proposition \ref{construction3}, $K'$ is geometrically $C_1$ and therefore $X(K')\neq \emptyset$. By the valuative criterion of properness, we get that $X(\Oo_{K'})\neq\emptyset$ and thus $X(k)\neq \emptyset$.
\end{proof}


\bt \label{c1transfer}
Given $d\in \N$, there is  a finite extension $E/\Q_p^{ur}$ such that every untilt of $\overline{\F}_p(\!(t^{1/p^{\infty}})\!)$ containing $E$ is geometrically $C_1$ up to dimension $d$. 
\et 
\begin{proof}
Let $\varphi_d$ express the property \quotes{geometrically $C_1$ up to dimension $d$} as in Proposition \ref{sentencephi}. Let $K'$ be a tame field containing $\overline{\Q_p}$ with divisible value group and residue field $\overline{\F}_p(t)^{h, \text{perf}}$, obtained via the construction of Proposition \ref{construction3}(ii) with $K_0=\overline{\Q_p}$. By the same proposition, the sentence $\varphi_d$ holds in $K'$.  By Fact \ref{perfectJK}, we know that $\overline{\F}_p(t)^{h,\text{perf}}$ is elementarily equivalent to $\overline{\F}_p(\!(t^{1/p^{\infty}})\!)$. Applying Proposition \ref{perfprop} for $F=\overline{\F}_p(\!(t^{1/p^{\infty}})\!)$ and $K'$ as above, there is a finite extension $E/\Q_p^{ur}$ (the compositum of the field $E$ from \ref{perfprop} with $\Q_p^{ur}$) such that $\varphi_d$ holds in every untilt of $\overline{\F}_p(\!(t^{1/p^{\infty}})\!)$ containing $E$. Thus, every such untilt is geometrically $C_1$ up to dimension $d$.
\end{proof}

\bc \label{immediate}
Given $d\in \N$, there is  a finite extension $E/\Q_p^{ur}$ such that $\widehat{E(\pi^{1/p^{\infty}})}$ is geometrically $C_1$ up to dimension $d$, for any uniformizer $\pi$ of $E$.
\ec 
\begin{proof}
Let $E$ be as in Theorem \ref{c1transfer} and $K=\widehat{E(\pi^{1/p^{\infty}})}$. A standard computation shows that 
$$\Oo_K/(\pi)\cong \Oo_E[X_1,X_2...]/(\pi,X_{1}^p-\pi,X_2^p-X_1,...)\cong \overline{\F}_p[t^{1/p^{\infty}}]/(t)$$
It follows that $\varprojlim_{\varphi} \Oo_K/(\pi) \cong \overline{\F}_p[\![t^{1/p^{\infty}}]\!]$ and hence $K^{\flat} \cong \overline{\F}_p(\!(t^{1/p^{\infty}})\!)$.
\end{proof}

\bc \label{concrete1}
Given $d\in \N$, there is  a finite extension $E/\Q_p^{ur}$ such that  $E(\pi^{1/p^{\infty}})$ is geometrically $C_1$ up to dimension $d$, for any uniformizer $\pi$ of $E$.
\ec
\begin{proof}
One knows, e.g., by van den Dries' Theorem 2.1.5 \cite{KK1}, that $E(\pi^{1/p^{\infty}})$ has the same theory as its $p$-adic completion. Thus we may apply Corollary \ref{immediate}.
\end{proof}

\subsection{Conditional results} \label{conditional}


Given a valued field $(K,v)$ and a proper $K$-variety $X$, an $\Oo_K$\textit{-model} of $X$ is a flat, proper $\Oo_K$-scheme whose generic fiber is isomorphic to $X$. By the valuative criterion of properness, a $K$-point of $X$ induces a $k$-point on each $\Oo_K$-model. The following result gives a converse in case $(K,v)$ is sufficiently nice.
\begin{fact} [Theorem 1.2.2 \cite{KK4}]\label{criterion}
Let $(K,v)$ be a tame field with divisible value group of rank $1$ and residue field $k$. Let $X$ be a proper $K$-variety. Suppose that $\Xx(k)\neq \emptyset$ for each $\Oo_K$-model $\Xx$ of $X$. Then $X(K)\neq \emptyset$.
\end{fact}  
The (conditional) transfer principle below essentially follows from \cite{KK4} but we spell out the argument for the convenience of the reader. 
\bt \label{geomc1trans}
Assume Conjecture \ref{kollcoll}. Let $(K,v)$ be a tame field with divisible value group and 
geometrically $C_1$ residue field $k$. Then $K$ is geometrically $C_1$.
\et   
\begin{proof}
By Proposition 3.1.5 \cite{KK4}, geometrically $C_1$ fields form an elementary class. By Lemma 5.1.3 \textit{loc.cit}, we have that $(K,v)$ has an elementary valued subfield $(K',v')$ such that $\Oo_{K'}$ is a colimit of DVRs. Without loss of generality, we can therefore assume that $\Oo_K$ itself is a colimit of DVRs. Let $X$ be a smooth projective separably rationally  connected $K$-variety. Assuming Conjecture \ref{kollcoll}, for every $\Oo_K$-model $\Xx$ of $X$, there is a smooth projective separably rationally connected $k$-variety $Y$ and a $k$-morphism $Y\to \Xx_k$. Since $k$ is geometrically $C_1$, we get that $Y(k)\neq \emptyset$ and hence  $\Xx_k(k)\neq \emptyset$. Since this is true for every $\Oo_K$-model of $X$, Theorem 1.2.2 \textit{loc.cit} gives that $X(K)\neq \emptyset$.
\end{proof} 


\bt \label{conditionperf2}
Assume Conjecture \ref{kollcoll}. Let $K$ be a perfectoid field with tilt $K^{\flat}$ which is geometrically $C_1$. Then $K$ is geometrically $C_1$.
\et  
\begin{proof}
Using Fact \ref{JKimproved}, we have the following
$$K^{\flat} \mbox{ geom. }C_1 \iff k_w \mbox{ geom. }C_1 \stackrel{\ref{geomc1trans}}  \Rightarrow K_U \mbox{ geom. }C_1 \iff  K \mbox{ geom. }C_1 $$
where the first and last equivalences use that geometrically $C_1$ fields form an elementary class and that $K^{\flat}\equiv k_w$ and $K_U\equiv K$ respectively.
\end{proof}
\bp \label{dejongcompl}
The field $\overline{\F}_p(\!(t^{1/p^{\infty}})\!)$ is geometrically $C_1$.
\ep 
\begin{proof}
By Fact \ref{perfectJK}, the field  $\overline{\F}_p(\!(t^{1/p^{\infty}})\!)$ has the same theory as the perfect hull of $\overline{\F}_p(t)^h$. The latter is a colimit of function fields of curves over $\overline{\F}_p$ and is therefore geometrically $C_1$ by de Jong-Starr \cite{dejongstarr}. Since geometrically $C_1$ fields form an elementary class, the same is  true for $\overline{\F}_p(\!(t^{1/p^{\infty}})\!)$.
\end{proof}
\begin{rem}
One could also argue as in the proof of Th\'eor\`eme 7.5 \cite{colliot}, using a generalized Greenberg approximation theorem due to Moret-Bailly \cite{MB} to reduce to the function field case.
\end{rem}

\bc
Assume Conjecture \ref{kollcoll}. Then every untilt of $\overline{\F}_p(\!(t^{1/p^{\infty}})\!)$ is geometrically $C_1$.
\ec
\begin{proof}
From Theorem \ref{conditionperf2} and Proposition \ref{dejongcompl}.
\end{proof}


\bc 
Assume Conjecture \ref{kollcoll}. Then $\Q_p^{ur}(p^{1/p^{\infty}})$ is geometrically $C_1$.
\ec
\begin{proof}
As in the proof of Corollary \ref{concrete1}.
\end{proof}
\begin{rem}
For the last two results, we only need the special case of Conjecture \ref{kollcoll} with $k$ being the function field of a curve over an algebraically closed field. As noted earlier, Starr proved this under an extra assumption on the DVR, so perhaps this special case of Conjecture \ref{kollcoll} is within reach. 
\end{rem}

\section{Perfectoid $C_2$ transfer}\label{C2}
We first prove our main theorem about untilts of $\F_p(\!(t^{1/p^{\infty}})\!)$. We then present a generalization in the context of rationally simply connected varieties, building on work of Starr-Xu.
\subsection{Perfectoid $C_2$ transfer} \label{perfectoidc2}
We recall the following theorem, whose proof relies crucially on earlier work of Esnault \cite{esnaultdeligne,esnault2}.
\begin{fact}[Theorem 1.1.2 \cite{KK4}] \label{factkk4}
Every tame valued field with divisible value group and finite residue field is $C_1$. 
\end{fact}
We also state Lang's transition theorem for the $C_i$ property:
\begin{fact}[\cite{langc1}]
Let $k$ be a $C_i$ field and $l/k$ be a field extension with $\text{tr.deg}(l/k)=s$. Then $l$ is $C_{i+s}$.
\end{fact}

\bp \label{construction4}
Let $(K,v)$ be a tame field with divisible value group and residue field $k$ which has transcendence degree $1$ over $\F_p$.

\begin{enumerate}[label=(\roman*)] 

\item If $(K,v)$ is of equal characteristic, then $K$ is $C_2$.

\item If $(K,v)$ is of mixed characteristic and $K\cap \overline{\Q}$ is defectless with divisible value group, then $K$ is $C_2$.
\end{enumerate}
\ep  
\begin{proof}
We proceed as in Proposition \ref{construction3}. Suppose $k$ is an algebraic extension of $\F_p(t)$.\\
(i) For any $k$ as above, it suffices to construct one such tame field which is $C_2$, as any other will have the same theory by Fact \ref{akekuhl}. Let $\F_p(\!(z^{\Q})\!)$ be the Hahn series field with value group $\Q$ and residue field $\F_p$. By Fact \ref{factkk4}, we have that $\F_p(\!(z^{\Q})\!)$ is $C_1$.
Consider the rational functon field $\F_p(\!(z^{\Q})\!)(t)$ equipped with the Gauss valuation $u$, extending the $z$-adic valuation on $\F_p(\!(z^{\Q})\!)$, with value group $\Q$ and residue field $\F_p(t)$. By Lemma \ref{construction2}, there is an algebraic extension $(K,v)$ of $(\F_p(\!(z^{\Q})\!)(t),u)$ such that $(K,v)$ is tame with value group $\Q$ and residue field $k$. Since the extension $K/\F_p(\!(z^{\Q})\!)$ is of transcendence degree $1$ and $\F_p(\!(z^{\Q})\!)$ is $C_1$, we get that $K$ is $C_2$ by Lang's transition theorem. \\
(ii) We equip $K_0=K\cap \overline{\Q}$ with the valuation $v_0$ which is an extension of the $p$-adic valuation on $\Q$ to $K_0$. By Fact \ref{kuhlake}, any two tame fields extending $(K_0,v_0)$ with divisible value group and residue field $k$ have the same theory. Now proceed as in (i), with $\F_p(\!(z^{\Q})\!)$ replaced by $K_0$, to construct one such tame field $(K,v)$ which is $C_2$. Any other such tame field will then also be $C_2$. 
\end{proof}

\bt \label{secondmainagain}
Let  $d\in \N$ and $E_{\infty}$ be a maximal totally ramified extension of $\Q_p$. Then there exists a finite subextension $E/\Q_p$ of $E_{\infty}/\Q_p$ such that every untilt of $\F_p(\!(t^{1/p^{\infty}})\!)$ containing $E$ is $C_2(d)$. 
\et
\begin{proof}
We equip $E_{\infty}(t)$ with the Gauss valuation $u$, with value group $\Q$ and residue field $\F_p(t)$. Let $(K',v')/(E_{\infty}(t),u)$ be an algebraic extension such that $(K',v')$ is tame with value group $\Q$ and residue field $\F_p(t)^{h,\text{perf}}$. By Proposition \ref{construction4}, we see that $K'$ is $C_2$. By Fact \ref{perfectJK}, we have that $\F_p(t)^{h,\text{perf}}\equiv \F_p(\!(t^{1/p^{\infty}})\!)$. We conclude from Proposition \ref{perfprop}.
\end{proof}

\bc \label{immediate2}
Given $d\in \N$, there is  a finite totally ramified extension $E/\Q_p$ such that $\widehat{E(\pi^{1/p^{\infty}})}$ is $C_2(d)$, for any uniformizer $\pi$ of $E$.
\ec 
\begin{proof}
As in the proof of Corollary \ref{immediate}.
\end{proof}

\bc \label{concrete}
Given $d\in \N$, there is  a finite totally ramified extension $E/\Q_p$ such that $E(\pi^{1/p^{\infty}})$ is $C_2(d)$, for any uniformizer $\pi$ of $E$.
\ec
\begin{proof}
As in the proof of Corollary \ref{concrete1}.
\end{proof}

We recall the first counterexample to Artin's conjecture by Terjanian \cite{Terj1} and observe that it admits rational points over every untilt of $\F_p(\!(t^{1/p^{\infty}})\!)$:
\begin{example}
Define 
$$G(\mathbf{X})=\sum_{i=1}^3 X_i^4- \sum_{i<j} X_i^2X_j^2-X_1X_2X_3(X_1+X_2+X_3)$$
where $\mathbf{X}=(X_1,X_2,X_3)$. Consider the quartic $\Xx\subseteq \mathbb{P}^{17}_{\Z_2}$ below
$$G(\mathbf{X})+G(\mathbf{Y})+G(\mathbf{Z})+4G(\mathbf{U})+4G(\mathbf{V})+4G(\mathbf{W}) =0$$
Let $\Xx_s=\Xx\times_{\Z_2} \F_2$ be the special fiber.
Terjanian observed that $\Xx(\Z/16\Z)=\emptyset$, which implies that $\Xx(\Q_2)= \emptyset$. One obstruction to the existence of a $\Q_2$-rational point is the fact that $\Xx_s^{\text{sm}}(\F_2)=\emptyset$. On the other hand, one checks that 
$$(1,1,t,1,1,t,0,...,0) \in \Xx_s^{\text{sm}}(\F_2(t))$$ 
By Hensel's Lemma, for any henselian valued field $K\supseteq \Q_p$ whose residue field contains $\F_p(t)$, we get that $\Xx(K)\neq \emptyset$. This applies to non-principal ultrapowers of untilts of $\F_p(\!(t^{1/p^{\infty}})\!)$ and hence $\Xx(K)\neq \emptyset$ for any such untilt. 
\end{example}
While no finite extension of $\Q_p$ is $C_2$ (or even $C_i$), due to Alemu \cite{alemu}, the degree of the counterexamples in \textit{loc.cit} grows with the degree of the extension. 
\bq 
Given $d\in \N$, is there a finite extension $E/\Q_p$ which is $C_2(d)$?
\eq 

\subsection{RSC varieties over global function fields} \label{RSC}
Informally, a variety $X$ is \textit{rationally simply connected} (RSC) if the space of 2-pointed rational curves on $X$ is rationally connected. This is analogous to the topological notion of simple connectedness, just as rational connectedness is analogous to path-connectedness. 
In our case, we need the notion of a \textit{rationally simply connected fibration}, see Hypothesis 1.2 \cite{starrxu} for the precise definition.

The main theorem of Starr-Xu says that a projective scheme over a global function field has a rational point if it deforms to a rationally simply connected variety in characteristic $0$ with vanishing elementary obstruction. 
More precisely:
\begin{fact} [Theorem 1.4 \cite{starrxu}] \label{starrxumain}
Let $R$ be a henselian DVR with fraction field $K$ of characteristic $0$ and finite residue field $k$.  Let $C_R$ be a generically smooth $R$-curve. Let $f_R:X_R\to C_R$ be a flat, projective morphism such that $f_K:X_K \to C_K$ is a rationally simply connected fibration whose generic fiber has vanishing elementary obstruction. 
Then, for every generic point $\eta$ of $C_k$ contained in the $R$-smooth locus of $C_R$, the pullback $X_{\eta}=X_R\times_{C_R} k(\eta)$ has a $k(\eta)$-rational point.
\end{fact} 

Starr notes that there is nothing special about finite fields in the proof, apart from the fact that they are \textit{RC solving} (due to Esnault \cite{esnault2} and Esnault-Xu \cite{esnaultxu}), i.e., they admit rational points in degenerations of rationally connected varieties. 
\begin{fact} [Proposition 4.3 \cite{starr2}]
Let $R$ be a henselian prime regular DVR with fraction field $K$ of characteristic $0$ and RC solving residue field $k$.  Let $C_R$ be a generically smooth $R$-curve. Let $f_R:X_R\to C_R$ be a flat projective morphism such that $f_K:X_K \to C_K$ is a rationally simply connected fibration whose generic fiber has vanishing elementary obstruction. Then, for every generic point $\eta$ of $C_k$ contained in the $R$-smooth locus of $C_R$, the pullback $X_{\eta}=X_R\times_{C_R} k(\eta)$ has a $k(\eta)$-rational point.
\end{fact} 
For our purposes, we need a version of the above result which is expressible in first-order logic. 
This was already done by Starr-Xu, although it was only documented in the first arXiv version \cite{starrxu1}. Using this, they gave an extension of the Ax-Kochen theorem in the context of rationally simply connected varieties, see Corollary 1.14 \cite{starrxu1}. We provide a brief overview in the next section. 

%
\subsection{First-order version} \label{first-order}
The idea is to work with a family $f:X\to M$ defined over $\Q$ which satisfies some rational simple connectedness condition. 
The definitions below can be made over a more general base but we only state them over $\Q$ for simplicity.
\begin{definition}[Definition 4.1 \cite{starr2}]
A \textit{parameter datum} over $\Q$ 
with a codimension $>1$ compactification is a datum 
$$((M, f: X \to M,\mathcal{L}_M ),(\overline{M}, \Oo_{\overline{M}}(1), \iota)) $$ 
where: 
\begin{enumerate}[label=(\roman*)]

\item $M$ is a smooth $\Q$-variety.
\item $f:X \to M$ is a proper, flat morphism.
\item $\mathcal{L}_M$ is an $f$-very ample invertible sheaf on $X$.
\item $\overline{M}$ is a proper $\Q$-variety.
\item $\iota: M\to \overline{M}$ is  a dense open immersion and $\overline{M}\backslash M$ has codimension $>1$.
\end{enumerate}

\end{definition}

For every $r \geq 1$, the complete linear system $H^0
(\overline{M}, \Oo_{\overline{M}}(r))$ induces a closed immersion into $\mathbb{P}^{N_r}_{\Q}$ where $N_r=h^0
(\Oo_{\overline{M}}(r))-1$. Let $\text{Gr}(N_r,m-1)$ be the Grassmannian variety parameterizing linear subspaces of $\mathbb{P}^{N_r}_{\Q}$ of codimension $m - 1$ and $G_r\subseteq \text{Gr}(N_r,m-1)$ be the open subset parameterizing those subspaces whose intersection with $\overline{M}$ is a smooth, irreducible curve which is contained in $M$.
\begin{definition}[Definition 4.2 \cite{starr2}]
A parameter datum as above satisfies the \textit{RSC property} if there is $r_0\in \N$ and a sequence $(W_r)_{r\geq r_0}$ of dense open subschemes $W_r\subseteq G_r$ such that for every algebraically closed field  $K$ of characteristic $0$ and every $K$-point of $W_r$ parametrizing a curve $C_K\subset M_K$, the pullback family $X\times_M C_K\to C_K$ is a rationally simply connected fibration whose generic fiber has vanishing elementary obstruction.
\end{definition}
The next result is the last assertion of Proposition 4.4 \cite{starr2}: 
\begin{fact}  [Proposition 4.4 \cite{starr2}]\label{standardmodel}
Consider a parameter datum as above satisfying the RSC property. Let $k$ be an RC solving field and $C$ be a geometrically integral $k$-curve. Then the map
$$f:X(k(C))\to M(k(C))$$
is surjective.  
\end{fact} 


\subsection{RSC varieties over perfectoid fields} \label{RSCperf}
\bt  \label{transferofsurj}
Let $f:X\to M$ be a finite type morphism of $\Q$-varieties. Let $F$ be a perfectoid field of characteristic $p$ and $E_{\infty}$ be a maximal totally ramified extension of $\Q_p$.  Then there exists a finite subextension $E/\Q_p$ of $E_{\infty}/\Q_p$ such that for every untilt $K$ of $F$ containing $E$, we have that  
$$f:X(K)\to M(K)$$ 
is surjective if and only if 
$$f:X(K')\to M(K')$$
is surjective, for any tame field $K'$ extending $E$ with divisible value group and residue field elementarily equivalent to $F$.
\et 
\begin{proof}
Since $X,M$ and $f$ are of finite type over $\Q$, there is an $\forall \exists$-sentence $\varphi$ in the language of rings such that for any field $K$ of characteristic $0$, we have that 
$$\varphi \mbox{ holds in } K\iff f:X(K)\to M(K) \mbox{ is surjective}$$
Applying Proposition \ref{perfprop} to $\varphi$ yields the desired conclusion.
\end{proof}

\bt \label{generalc2}
Consider a parameter datum satisfying the RSC property. Let $E_{\infty}$ be a maximal totally ramified extension of $\Q_p$. Then there exists a finite subextension $E/\Q_p$ of $E_{\infty}/\Q_p$ such that for every untilt $K$ of $\F_p(\!(t^{1/p^{\infty}})\!)$ containing $E$, the map 
$$f:X(K)\to M(K)$$ 
is surjective.
\et 
\begin{proof}
By Corollary 1.1.3 \cite{KK4}, we know that $E_{\infty}$ is geometrically $C_1$. By Hogadi-Xu \cite{hogadi} (or Corollary 3.2.12 \cite{KK4}), we get that $E_{\infty}$ is also RC solving. By Fact \ref{standardmodel}, we know that 
$$f:X(E_{\infty}(C))\to M(E_{\infty}(C))$$
is surjective for any geometrically integral $E_{\infty}$-curve $C$.
Now let $(K',v')$ be as in the proof of Theorem \ref{perfectoidc2}. Since $K'$ is algebraic over $E_{\infty}(t)$ and $\overline{E_{\infty}}\cap K=E_{\infty}$,  we see that $K'$ is a colimit of function fields of geometrically integral $E_{\infty}$-curves. It easily follows that 
$$f:X(K')\to M(K')$$
is surjective. We conclude from Theorem \ref{transferofsurj}.
\end{proof}
Theorem \ref{secondmainagain} is a special case of the above theorem by taking $f:X \to M$ the universal family of hypersurfaces in $\mathbb{P}^n$ of degree $d$. If $n\geq d^2$, then the RSC condition is satisfied by the last assertion of Theorem 3.16 \cite{starrxu}. For any field $K$, the surjectivity of $f:X(K)\to M(K)$ means precisely that every hypersurface of degree $d$ with $n\geq d^2$ has a $K$-rational point, i.e., that $K$ is $C_2(d)$.

\subsubsection*{Acknowledgements} 
I wish to thank H{\'e}l{\`e}ne Esnault, Franziska Jahnke and Peter Scholze for pointing me towards the questions addressed in this paper. I also thank Jean-Louis Colliot-Th{\'e}l{\`e}ne, Philip Dittmann and Tom Scanlon for fruitful discussions on these topics. Finally, I am grateful to the anonymous referees for their thoughtful comments.

\bibliographystyle{alpha}
\bibliography{references}
\end{document}